\documentclass{amsart}
\usepackage{amsmath,amssymb,color}
\usepackage{amsthm}

\newtheorem{theorem}{Theorem}

\newtheorem{proposition}[theorem]{Proposition}
\newtheorem{remark}[theorem]{Remark}

\newcommand {\N}{{\bf N}}

\def\re{\mathbb{R}}

\def\N{\mathbb{N}}

\def\e{\varepsilon}

\def\pd{\partial}
\def\ol{\overline}

\def\la{\lambda}
\def\disp{\displaystyle}

\def\({\left(}
\def\){\right)}

\def\pd{\partial}

\def\intO{\int_{\Omega}}

\def\phia{\phi_{\alpha}}

\begin{document}
\title{Hardy's inequality in a limiting case on general bounded domains}

\author{Jaeyoung Byeon \and Futoshi Takahashi}

\address{%
Department of Mathematical Sciences, KAIST \\
291 Daehak-ro, Yuseong-gu,  Daejeon 34141,
Republic of Korea \\
and \\
Department of Mathematics, Osaka City University \\
3-3-138 Sugimoto, Sumiyoshi-ku, Osaka 558-8585, Japan}

\email{byeon@kaist.ac.kr \\ futoshi@sci.osaka-cu.ac.jp}

\begin{abstract}
In this paper, we study Hardy's inequality in a limiting case:
$\disp{\intO |\nabla u |^N dx \ge C_N(\Omega) \intO \frac{|u(x)|^N}{|x|^N \( \log \frac{R}{|x|} \)^N} dx}$
for functions $u \in W^{1,N}_0(\Omega)$, where $\Omega$ is a bounded domain in $\re^N$ with $R = \sup_{x \in \Omega} |x|$.
We study the attainability of the best constant $C_N(\Omega)$ in several cases.
We provide sufficient conditions that assure $C_N(\Omega) > C_N(B_R)$ and $C_N(\Omega)$ is attained, here $B_R$ is the $N$-dimensional ball with center the origin and radius $R$.
Also we provide an example of $\Omega \subset \re^2$ such that $C_2(\Omega) > C_2(B_R) = 1/4$ and $C_2(\Omega)$ is not attained.
\end{abstract}

\subjclass[2010]{Primary 35A23; Secondary 26D10.}

\keywords{Hardy's inequality in a limiting case, best constants.}
\date{\today}

\dedicatory{}

\maketitle

\section{Introduction}

The classical Hardy inequality in one space dimension states that
\begin{equation}
\label{Hardy_1D}
	\int_0^{\infty} |u'(t)|^p \, dt \ge \( \frac{p-1}{p} \)^p \int_0^{\infty} \frac{|u(t)|^p}{t^p} \, dt
\end{equation}
holds for all $u \in W^{1,p}_0(0, +\infty)$ where $1 < p < \infty$.
This scaling invariant inequality is now  very classical and there are  wonderful treatises  \cite{Ghoussoub-Moradifam(book)}, \cite{Mazya}, \cite{Opic-Kufner} on further generalizations of the inequality \eqref{Hardy_1D}.
It is also known that the constant $\( \frac{p-1}{p} \)^p$ is best possible and it is not achieved by any function in $W^{1,p}_0(0,+\infty)$.
The inequality \eqref{Hardy_1D} has been generalized to higher dimensions in two directions:
one is to replace the function $t$ in the right-hand side by the distance to the origin,
and the other is to replace it by the distance to the boundary.

For the former direction, let $\Omega$ be a domain with $0 \in \Omega$ in $\re^N$ ($N \ge 2$) and let $p \geq 1$.
Then the classical $L^p$-Hardy inequality states that
\begin{equation}
\label{H_p}
	\intO |\nabla u|^p \, dx \ge \left| \frac{N-p}{p} \right|^p \intO \frac{|u|^p}{|x|^p} \, dx
\end{equation}
holds for all $u \in W^{1,p}_0(\Omega)$ when $1 \le p < N$, and
for all $u \in W^{1,p}_0(\Omega \setminus \{ 0 \})$ when $p > N$.
It is known that for $p > 1$, the best constant $|\frac{N-p}{p}|^p$ is never attained in $W^{1,p}_0(\Omega)$ when $p < N$, or in $W^{1,p}_0(\Omega \setminus \{ 0 \})$ when $p > N$, respectively.
After the pioneering work of Brezis and V\'{a}zquez \cite{Brezis-Vazquez}, which showed that the inequality can be improved on bounded domains when $p < N$,
there are many papers that treat the improvements of the inequality (\ref{H_p})
(see \cite{ACR}, \cite{BFT1}, \cite{BFT2}, \cite{Cazacu}, \cite{DPP}, \cite{Filippas-Tertikas}, \cite{GGM}, \cite{Sano-TF},
the recent book \cite{Ghoussoub-Moradifam(book)}  and the reference therein.)

For the latter direction,
let $\Omega \subset \re^N$ be an open set with Lipschitz boundary and define $d(x) = {\rm dist}(x, \pd\Omega)$.
Then, a version of Hardy inequalities, called ``geometric type", states that for any $p > 1$,
there exists $c_p(\Omega) > 0$ such that the inequality
\begin{equation}
\label{GH_p}
	\intO |\nabla u|^p \, dx \ge c_p(\Omega) \intO \frac{|u|^p}{(d(x))^p} \, dx
\end{equation}
holds for all $u \in W^{1,p}_0(\Omega)$. 
For this inequality, refer to \cite{Ancona}, \cite{BFT1}, \cite{Brezis-Marcus}, \cite{DPP}, \cite{LP}, \cite{Lehrback}, \cite{MS(NA)}, \cite{Tidblom(JFA)}, \cite{Tidblom(PAMS)}, 
the recent book \cite{BEL(book)} and the references therein.
In \cite{MS(NA)}, it is proved that $c_p(\Omega) = \( \frac{p-1}{p} \)^p$ is the best constant on any convex domain $\Omega$,
that is,
\begin{equation}
\label{hq}
	c_p(\Omega) = \inf_{u \in W^{1,p}_0(\Omega), u \not\equiv 0}
	\frac{\intO |\nabla u |^p dx}{\intO \frac{|u(x)|^p}{(d(x))^p} dx} = \( \frac{p-1}{p} \)^p.
\end{equation}
In \cite{BFT1}, \cite{Tidblom(PAMS)}, the authors obtained an additional extra term on the right-hand side of (\ref{GH_p}),
which means that the best constant $c_p(\Omega)$ is never attained on any convex domain $\Omega$.
When $\Omega$ is the half-space $\re^N_{+} = \{ x = (x_1, \cdots, x_N) \,|\, x_N > 0 \}$, the inequality (\ref{GH_p}) has the form
\begin{equation}
\label{Hardy_half}
	\int_{\re^N_{+}} |\nabla u|^p \, dx \ge \( \frac{p-1}{p} \)^p \int_{\re^N_{+}} \frac{|u|^p}{x_N^p} \, dx
\end{equation}
and the best constant $\( \frac{p-1}{p} \)^p$ is never attained by functions in $W^{1,p}_0(\re^N_{+})$.
On the other hand, let $\Omega$ be a bounded domain with $C^{1, \gamma}$ boundary for some $\gamma \in (0,1)$.
Then it is proved by Marcus, Mizel, and Pinchover in \cite{MMP} that
there exists a minimizer of $C_2(\Omega)$ if and only if $C_2(\Omega) < 1/4$.
See also \cite{MMP}, \cite{Marcus-Shafrir}, \cite{LP} for the corresponding results for $1 < p < \infty$. 
So the compactness of any minimizing sequence fails only at the
bottom level $\( \frac{p-1}{p} \)^p.$

%
%
In the critical case $p = N$, the weight $|x|^{-N}$ is too singular for the same type of inequality as (\ref{H_p}) to hold true for functions in $W^{1,N}_0(\Omega)$.
Instead of (\ref{H_p}), it is known that the following {\it Hardy inequality in a limiting case}
\begin{equation}
\label{Hardy_N}
	\intO |\nabla u |^N dx \ge \( \frac{N-1}{N} \)^N \intO \frac{|u(x)|^N}{|x|^N \( \log \frac{R}{|x|} \)^N} dx
\end{equation}
holds true for all $u \in W^{1,N}_0(\Omega)$ where $R = \sup_{x \in \Omega} |x|$;
refer to \cite{Leray}, \cite{Ladyzhenskaya}, \cite{DP}, \cite{Ioku-Ishiwata}, \cite{TF} and references therein.
Note that the additional $\log$ term weakens the singularity of $|x|^{-N}$ at the origin,
however, the weight function
\[
	W_R(x) = \frac{1}{|x|^N \( \log \frac{R}{|x|} \)^N}
\]
becomes singular also on the boundary $\pd\Omega$ since $R = \sup_{x \in \Omega} |x|$.
Indeed, since
\begin{equation}
\label{Taylor}
	|x|^N \( \log \frac{R}{|x|} \)^N = (R-|x|)^N + o((R-|x|)^N)
\end{equation}
as $|x| \to R$, $W_R$ has a similar effect of $(1/d(x))^N$ near the boundary.
In this sense, the critical Hardy inequality (\ref{Hardy_N}) has both features of the inequalities \eqref{H_p} and \eqref{GH_p}.
Note that (\ref{Hardy_N}) is invariant under the scaling
\begin{equation}
\label{scaling_N}
	u_{\la}(x) = \la^{-\frac{N-1}{N}} u\( \( \frac{|x|}{R} \)^{\la-1} x \) \quad \text{for} \, \la > 0,
\end{equation}
which is different from the usual scaling $u_{\la}(x) = \la^{\frac{N-p}{p}} u(\la x)$ for (\ref{H_p}) when $\Omega = \re^N$ and $p < N$.
(However recently, a relation of both scaling transformations is obtained, see \cite{Sano-TF}).

%
%
Let $C_N(\Omega)$ be the best constant of the inequality (\ref{Hardy_N}):
\begin{equation}
\label{CHN}
	C_N(\Omega) = \inf_{u \in W^{1,N}_0(\Omega), u \not\equiv 0}
	\frac{\intO |\nabla u |^N dx}{\intO \frac{|u(x)|^N}{|x|^N \( \log \frac{R}{|x|} \)^N} dx}.
\end{equation}
By this definition and \eqref{Hardy_N}, we see $C_N(\Omega) \ge \( \frac{N-1}{N} \)^N$ for any bounded domain $\Omega \subset B_R$ with $R = \sup_{x \in \Omega} |x|$.
Here and henceforth, $B_R$ will denote the $N$-dimensional ball with radius $R$ and center $0$.

In \cite{Ioku-Ishiwata}, the authors proved that $C_N(B_R) = \( \frac{N-1}{N} \)^N$ and $C_N(B_R)$ is never attained by any function in $W^{1,N}_0(B_R)$.
See also \cite{DFP}, \cite{DP}.
Let us recall the arguments in \cite{Ioku-Ishiwata}.
First, the authors of \cite{Ioku-Ishiwata} prove that, if the infimum $C_N(B_R)$ is attained by a radially symmetric function $u \in W^{1,N}_{0, rad}(B_R)$,
then $u \in C^1(B_R \setminus \{ 0 \})$, $u > 0$ and $u$ is unique up to multiplication of positive constants.
By using these facts and the scaling invariance (\ref{scaling_N}), the authors prove that $C_N(B_R)$ is not attained by radially symmetric functions.
Indeed, by the scaling invariance (\ref{scaling_N}) and the uniqueness up to multiplication of positive constants,
the possible radially symmetric minimizer has the form $C (\log \frac{R}{|x|})^{\frac{N-1}{N}}$ which is not in $W^{1,N}_0(B_R)$.
Finally, they prove that if there exists a minimizer of $C_N(B_R)$, then there exists also a radially symmetric minimizer.
The argument of this part is elementary and the proof of the non-attainability of $C_N(B_R)$ is established.

The main purpose of this paper is to study the (non-)attainability of the infimum $C_N(\Omega)$ for more general domains $\Omega \subset B_R$. 
Some new phenomena will be shown in this paper.
We first note that if $C_N(\Omega) = \( \frac{N-1}{N} \)^N$,  $C_N(\Omega)$ is not attained.
In fact, if $C_N(\Omega)$ is attained by an element $u \in W_0^{1,N}(\Omega),$ by a trivial extension of $u$ as an element in $W_0^{1,N}(B_R),$
$C_N(B_R) = \( \frac{N-1}{N} \)^N $ is attained  by $u$; this contradicts the result in \cite{Ioku-Ishiwata} that $ C_N(B_R) $ is not attained.
In the following, we may not impose the assumption that $0 \in \Omega$.
Since the weight function $W_R(x) = (|x| (\log \frac{R}{|x|} ))^{-N}$ itself depends on the geometric quantity $R$,
it is not clear whether $C_N(\Omega)$ has the same value as $C_N(B_R)$ for all domains $\Omega \subset B_R$ or not.
Since $W_R$ becomes unbounded around the origin and also around the set $|x| = R$,
it is plausible that minimizing sequences for $C_N(\Omega)$ tend to concentrate on the origin or on the boundary portion $\pd \Omega \cap \pd B_R$
in order to minimize the quotient
\[
	Q_R(u) = \frac{\intO |\nabla u |^N dx}{\intO W_R(x) |u(x)|^N dx}.
\]
This will result in that $C_N(\Omega) = C_N(B_R)$ and $C_N(\Omega)$ is not attained,
if the origin is the interior point of $\Omega$, or $\Omega$ has a smooth boundary portion at a distance $R$ to the origin
(just like a ball $B_R$).
We will prove later that these intuitions are true, see Theorem \ref{theorem-origin} and Theorem \ref{theorem-smooth}.
However, when we treat a domain $\Omega \subset B_R$ with $R = \sup_{x \in \Omega} |x|$,
which does not contain the origin in its interior, nor have the smooth boundary portion $\pd\Omega \cap B_R$,
the situation is rather different.
Actually, we provide a sufficient condition on $\Omega \subset B_R$ which assures that $C_N(\Omega) > C_N(B_R)$ (Theorem \ref{theorem-inequality}).
Moreover, we prove that a stronger condition on $\Omega$ than the sufficient condition assures  that $C_N(\Omega)$ is attained (Theorem \ref{theorem-existence}).
Finally, we provide an example of domain in $\re^2$ on which $C_2(\Omega) > C_2(B_R) = 1/4$ and $C_2(\Omega)$ is not attained  (Theorem \ref{theorem-nonexistence}).
This is quite a contrast to the result for \eqref{hq} in \cite{MMP}, which says that
if $c_2(\Omega)$ is strictly less than the critical number $\frac 14,$ the infimum $c_2(\Omega)$ is attained.

The organization of this paper is as follows:
In \S 2, we prove Theorem \ref{theorem-origin}, which says that if $0 \in \Omega$, then
$C_N(\Omega) =\( \frac{N-1}{N} \)^N$ and the infimum is not attained.
In \S 3, we prove Theorem \ref{theorem-smooth}, which says that if $\partial B_R \cap \partial \Omega$
enjoys some regularity, then $C_N(\Omega) =\( \frac{N-1}{N} \)^N$ and the infimum is not attained.
In \S 4, we prove Theorem \ref{theorem-inequality}, which says that a strict inequality
$C_N(\Omega) > \( \frac{N-1}{N} \)^N$ holds under some condition on $\Omega$
and Theorem \ref{theorem-existence}, which says that under a stronger condition than the one in Theorem \ref{theorem-inequality}, the infimum is attained.
Finally in \S 6, we prove Theorem \ref{theorem-nonexistence}, which says that the condition for the existence of a minimizer in Theorem \ref{theorem-existence} is optimal.

Now, we fix some notations and usages.
For a bounded domain $\Omega \subset \re^N$,
the letter $R$ will be used to denote $R = \sup_{x \in \Omega} |x|$ throughout the paper.
$B_R$ will denote the $N$-dimensional ball with radius $R$ and center $0$.
The surface area $\int_{S^{N-1}} dS_{\omega}$ of the $(N-1)$ dimensional unit sphere $S^{N-1}$ in $\re^N$ will be denoted by $\omega_{N-1}$.
$S^{N-1}(r)$ will denote the sphere of radius $r$ with center $0$.
Finally, the letter $C$ may vary from line to line.

%
%
\begin{section}{Hardy's inequality  for the case $0  \in \Omega$}

In this section, we treat the case when $\Omega \subset B_R$ has the origin as an interior point of $\Omega$.
In this case, we prove the following theorem.

\begin{theorem}
\label{theorem-origin}
For any bounded domain $\Omega \subset \re^N$ with $0 \in \Omega$ and $R = \sup_{x \in \Omega} |x|$,
\begin{align*}
	C_N(\Omega) = C_N(B_R) = \( \frac{N-1}{N} \)^N,
\end{align*}
and the infimum $C_N(\Omega)$ is not attained.
\end{theorem}

\begin{proof}
Note that by the definition of $R$, we have $\Omega \subset B_R$.
By a trivial extension of a function $u \in W^{1,N}_0(\Omega)$ on $B_R$ by $u(x) = 0$ for $x \in B_R \setminus \Omega$,
we see $W^{1,N}_0(\Omega) \subset W^{1,N}_0(B_R)$ and thus
\begin{equation}	
\label{C_N_lower}
	C_N(\Omega) \ge C_N(B_R) = \( \frac{N-1}{N} \)^N.
\end{equation}	
For the fact $C_N(B_R) = \( \frac{N-1}{N} \)^N$, we refer to \cite{Ioku-Ishiwata}.
In \cite{Ioku-Ishiwata}, the authors prove this fact by using the test functions
\begin{align*}
	\psi_{\beta} (x) = \begin{cases}
				1, &\quad 0 \le |x| \le \frac{R}{e}, \\
				\( \log \frac{R}{|x|} \)^\beta, &\quad \frac{R}{e} \le |x| \le R
			\end{cases}
\end{align*}
for $\beta > \frac{N-1}{N}$.
Note that $\{ \psi_{\beta} \}$ will concentrate on the boundary $\pd B_R$ when $\beta \downarrow \frac{N-1}{N}$.
In our case, since $0 \in \Omega$ is an interior point, there exists a small $c \in (0,1)$ such that $B_{cR}(0) \subset \Omega$.
For $0 < \alpha < \frac{N-1}{N}$, we define a function
\begin{align*}
	\phia (x) = \begin{cases}
				\( \log \frac{R}{|x|} \)^{\alpha}, &\quad |x| \le \frac{cR}{2}, \\
				\( \log \frac{2R}{c} \)^{\alpha}(2-\frac{2|x|}{cR}) , &\quad \frac{cR}{2} \le |x| \le  cR, \\
				0, &\quad  cR \le |x|, \, \text{and} \; x \in \Omega.
			\end{cases}
\end{align*}
Then we see that
\begin{align*}
	A \equiv &\intO |\nabla \phia|^N dx
	= \omega_{N-1} \int_0^{\frac{cR}{2}} \left| \alpha \( \log \frac{R}{r} \)^{\alpha-1} \( \frac{-1}{r} \) \right|^N r^{N-1} dr + O(1) \\
	&= \omega_{N-1} \alpha^N \int_0^{\frac{cR}{2}} \( \log \frac{R}{r} \)^{N(\alpha-1)} \frac{1}{r} \;  dr + O(1) \\
	&= \omega_{N-1} \alpha^N \left[ \frac{-1}{N(\alpha-1) + 1} \( \log \frac{R}{r} \)^{N(\alpha-1) + 1} \right]_0^{\frac{cR}{2}}  + O(1) \\
	&= \omega_{N-1} \alpha^N \( \frac{-1}{N(\alpha-1) + 1} \) \log \frac{2}{c} + O(1).
\end{align*}
Since $\alpha < \frac{N-1}{N}$, we have $N(\alpha-1) + 1 < 0$.
Thus $|\nabla \phia|^N$ is integrable near the origin and $\phia \in W^{1,N}_0(\Omega)$ for any $\alpha \in (0, \frac{N-1}{N})$.
Also we see that
 \begin{align*}
	B \equiv  &\intO \frac{|\phia(x)|^N}{|x|^N \( \log \frac{R}{|x|} \)^N} dx
	= \omega_{N-1} \int_0^{\frac{cR}{2}} \frac{(\log \frac{R}{r})^{\alpha N}}{r^N (\log \frac{R}{r})^N} r^{N-1} dr + O(1) \\
	&= \omega_{N-1} \int_0^{\frac{cR}{2}} \( \log \frac{R}{r} \)^{N\alpha - N} \frac{1}{r} \; dr + O(1) \\
	&= \omega_{N-1} \( \frac{-1}{N(\alpha-1) + 1} \) \log \frac{2}{c} + O(1).
\end{align*}
Therefore, we conclude that
\begin{align*}
	\frac{A}{B} &= \frac{\omega_{N-1} \alpha^N \( \frac{-1}{N(\alpha-1) + 1} \) \log \frac{2}{c} + O(1)}{\omega_{N-1} \( \frac{-1}{N(\alpha-1) + 1} \) \log \frac{2}{c} + O(1)}
	= \frac{\alpha^N  + O(1) (N(\alpha-1) + 1)}{1  + O(1) (N(\alpha-1) + 1)} \\
	&\to  \( \frac{N-1}{N} \)^N \quad \text{as} \; \alpha \uparrow \frac{N-1}{N}.
\end{align*}
This proves that
\[
	C_N(\Omega) = \( \frac{N-1}{N} \)^N,
\]
thus  the infimum $C_N(\Omega)$ is not attained; see Introduction.
\end{proof}
\end{section}

%
%
\begin{section}{Hardy's inequality  for smooth domains}

In this section, we prove that $C_N(\Omega)$ equals to $\( \frac{N-1}{N} \)^N$ if the domain has a smooth boundary portion on $\pd B_R$.
For the smoothness on the boundary, the interior sphere condition is enough to obtain the result.
Here we say that a point $x_0 \in \pd\Omega \cap \pd B_R$ satisfies an {\it interior sphere condition} if there is an open ball $B \subset \Omega$
such that $x_0 \in \pd B$.
The idea here is to construct a (non-convergent) minimizing sequence $\{ u_n \}$ for $C_N(\Omega)$ for which the value of $Q_R(u_n)$ goes to $\( \frac{N-1}{N} \)^N$,
by modifying a minimizing sequence for the best constant of Hardy's inequality on the half-space \eqref{Hardy_half} when $p = N$:
\begin{equation}
\label{Hardy_half_inf}
	\inf_{u \in C_0^\infty(\re^N_{+}) \setminus \{0\}} \frac{\int_{\re^N_{+}} |\nabla u|^N dx}{\int_{\re^N_{+}} |\frac{u}{x_N}|^N dx} = \( \frac{N-1}{N} \)^N.
\end{equation}
This is possible since the weight function $W_R(x)$ can be considered as $(1/d(x))^N$ near the smooth boundary portion $\pd\Omega \cap \pd B_R$.

\begin{theorem}
\label{theorem-smooth}
For a bounded domain $\Omega$, we assume that there exists a point $x_0 \in \pd\Omega \cap \pd B_R$ satisfying an interior sphere condition.
Then
\[
	C_N(\Omega) = \( \frac{N-1}{N} \)^N
\]
and the infimum $C_N(\Omega)$ is not attained.
\end{theorem}

\begin{proof}
The following proof is inspired by \cite{MMP}.
We write $x = (x_1, \cdots, x_{N-1}, x_N) = (x^{\prime}, x_N)$ for $x \in \re^N_{+}$.
Fix $\e > 0$ arbitrary.
By \eqref{Hardy_half_inf}, we may take $v_\e \in C_0^\infty(\re^N_+)$ such that
\[
	\int_{\re^N_+} \left|\frac{v_\e}{x_N} \right|^N dx = 1, \quad \text{and} \quad  \int_{\re^N_+} |\nabla v_\e|^N dx \le \(\frac{N-1}{N} \)^N + \e.
\]
Since $\textrm{supp}(v_\e)$ is compact, we may assume that
\[
	\textrm{supp}(v_\e) \subset \{x = (x^{\prime}, x_N) \in \re^N_+ \ | \  |x^{\prime}|^2 <  A x_N, \  x_N < B \}
\]
if we take $A,B > 0$ sufficiently large depending on $\e$.
We think $v_{\e}$ is $0$ outside of its support and is defined on the whole $\re^N_{+}$.
For $l \in \N$, we define $v_\e^l(x) = v_\e(l x)$.
Note that for each $l > 0$, we have
\[
	\int_{\re^N_+} |\nabla v^l_\e|^N dx = \int_{\re^N_+} |\nabla v_\e|^N dx, \quad  \int_{\re^N_+} \left|\frac{v^l_\e}{x_N} \right|^N dx = \int_{\re^N_+} \left|\frac{v_\e}{x_N} \right|^N dx
\]
and
\[
	\textrm{supp}(v^l_\e) \subset \left\{(x^{\prime}, x_N) \in \re^N_+ \ | \  |x^{\prime}|^2 <  \frac{A}{l} x_N, \  x_N < \frac{B}{l} \right\}.
\]
By a rotation, we may assume that $x_0 = (-R) e_N \in \partial \Omega \cap \pd B_R$ satisfies an interior sphere condition,
where $e_N = (0, \cdots, 0, 1)$.
Then we see that for some $A^{\prime}$, $B^{\prime} > 0$,
\[
	\{(x^{\prime}, x_N) \in \re^N_+ \ | \  |x^{\prime}|^2 <  A^{\prime}  x_N, \  x_N < B^{\prime} \} \subset \Omega + R e_N
\]
Since \eqref{Taylor} holds for small $R - |x|$, we see that
\begin{equation}
\label{S1}
	|x|^N \( \log\frac{R}{|x|} \)^N \le (x_N + R)^N + o((x_N +R)^N)
\end{equation}
for $x \in \Omega$ with small $x_N+R$.
Now we define
\[
	u_\e^l(x) \equiv v_\e^l(x + R e_N)
\]
for $x \in \Omega$.
Then, for large $l > 0$, we see that
$u_\e^l \in C_0^\infty(\Omega)$ and
\[
	{\rm supp}(u_\e^l) \subset \Omega \cap \{x \in B_R \ | \ x_N+R < B/l\}.
\]
Now \eqref{S1} implies that
\begin{align*}
	\intO \frac{|u_\e^l(x)|^N}{|x|^N\big (\log\frac{R}{|x|}\big )^N} dx \geq \intO \frac{|u_\e^l(x)|^N}{(x_N + R)^N} dx + o_l(1) = \int_{\Omega + Re_N} \frac{|v_\e^l(y)|^N}{|y_N|^N} dy +o_l(1)
\end{align*}
where $o_l(1) \to 0$ as  $l \to \infty$,
and
\begin{align*}
	\intO \big |\nabla u_{\e}^l(x) \big|^N dx = \int_{\Omega + Re_N} \big |\nabla v_{\e}^l(y) \big|^N dy \leq \int_{\re^N_{+}} \big |\nabla v_{\e}^l(y) \big|^N dy.
\end{align*}
Thus we have
\begin{align*}
	\frac{\intO \big |\nabla u_{\e}^l(x) \big|^N dx}{\intO \frac{|u_\e^l(x)|^N}{|x|^N\big (\log\frac{R}{|x|}\big )^N} dx}
	\le
	\frac{\int_{\re^N_+} \big |\nabla v_\e^l\big|^N dy}{\int_{\re^N_+} \frac{|v_\e^l(y)|^N}{|y_N|^N} dy} + o_l(1)
	\le  \(\frac{N-1}{N} \)^N + \e + o_l(1).
\end{align*}
This implies that
\[
	\inf_{u \in W_0^{1,N}(\Omega) \setminus \{0\}} \frac{\intO \big |\nabla u \big|^N dx}{\intO \frac{|u(x)|^N}{|x|^N \(\log\frac{R}{|x|}\)^N} dx}
	\le \( \frac{N-1}{N} \)^N.
\]
Since $C_N(\Omega) \ge C_N(B_R) = \(\frac{N-1}{N} \)^N$ by \eqref{C_N_lower}, we conclude the equality.
This again implies that the infimum $C_N(\Omega)$ is not attained.
\end{proof}
\end{section}

%
%

\begin{section}{Hardy's inequality for nonsmooth domains}

In this section, first we provide a sufficient condition to assure the strict inequality $C_N(\Omega) > C_N(B_R)$ for bounded domains $\Omega$ with $R = \sup_{x \in \Omega} |x|$.

First, we recall the notion of spherical symmetric rearrangement. 
Let $B_r(p,s)$ denote the geodesic open ball in $S^{N-1}(r)$ with center $p \in S^{N-1}(r)$ and geodesic radius $s$.
Then for each $r \in (0,R)$, there exists a constant $a(r) \ge 0$ such that
the $(N-1)$-dimensional measure of the geodesic open ball $B_r(r e_N, a(r))$ with center $r e_N = (0,\cdots,0,r)$ and radius $a(r)$ equals to $\mathcal{H}^{N-1}(\Omega \cap S^{N-1}(r))$,
here $\mathcal{H}^{N-1}$ denotes the $(N-1)$-dimensional Hausdorff measure.
Define the {\it spherical symmetric rearrangement} $\Omega^*$ of a domain $\Omega \subset B_R$ by
\[
	\Omega^* \equiv  \bigcup_{r \in (0,R)} B_r(r e_N, a(r))
\]
and the {\it spherical symmetric rearrangement} $u^*$ of a function $u$ on $\Omega$ by
\[
	u^*(x) \equiv \sup \{ t \in \re \, | \, x \in \{ x \in \Omega \, | \, u(x) \ge t \}^{*} \}, \quad x \in \Omega^*,
\]
see Kawohl \cite{Kawohl} p.17.
Note that this is an equimeasurable rearrangement with $u^*$ rotationally symmetric around the positive $x_N$-axis,
and there hold that the Polya-Szeg\"o type inequality
\[
	\intO |\nabla u|^p \, dx \ge \int_{\Omega^*} |\nabla u^*|^p \, dx
\]
for $u \in W^{1,p}_0(\Omega)$ with $p > 1$,
and the Hardy-Littlewood inequality
\[
	\intO u(x) v(x) \, dx \le \int_{\Omega^*} u^*(x) v^*(x) \, dx
\]
for nonnegative functions $u, v$ on $\Omega$, see \cite[pages 21, 23, and 26]{Kawohl}.

In the sequel, we use the {\it Poincar\'e inequality on a subdomain of spheres} of the following form:
\begin{proposition}
\label{prop-Poincare}
Let $S^n$ denote an $n$-dimensional unit sphere and $U \subset S^n$ be a relatively compact open set in $S^n$.
For any $1 \le p < \infty$,
there exists $C > 0$ depending on $p$ and $n$ such that the inequality
\[
	\int_U | \nabla_{S^n} u |^p dS_{\omega} \ge C | U |^{-p/n} \int_U |u|^p dS_{\omega}
\]
holds for any $u \in W^{1,p}_0(U)$.
Here $|U|$ denotes the $n$-dimensional measure of $U \subset S^n$.
\end{proposition}

\begin{proof}
The inequality $\int_U | \nabla_{S^n} u |^p dS_{\omega} \ge C(U, p)  \int_U |u|^p dS_{\omega}$ holds, see for example, \cite{Saloff-Coste} pp.86.
The constant $C(U, p)$ is bounded from below by the first Dirichlet eigenvalue $\lambda_p(U)$ of the $p$-Laplacian $-\Delta_p$ on the sphere,
and the estimate
\[
	\la_p(U) \ge C(n, p) |U|^{-p/n}
\]
can be seen, for example, in \cite{Lieb} or \cite{Kawohl-Fridman} when the ambient space is $\re^n$.
Indeed, the lower bound of the first Dirichlet eigenvalue is also obtained on spheres.
By spherically symmetric rearrangement, we have the Faber-Krahn type inequality
\[
	\la_p(U) \ge \la_p(U^*)
\]
where $U^* \subset S^n$ be a geodesic ball with $|U| = |U^*|$.
Also we have a scaling property $\la_p(r U) = r^{-p} \la_p(U)$ for the first eigenvalue of the $p$-Laplacian.
Since $U^* = r B_1$ for some $r > 0$ where $B_1$ denotes the geodesic ball of radius $1$, we have $|U| = |U^*| = r^n |B_1|$,
which implies $r = (|U|/|B_1|)^{1/n}$.
Thus we have
\[
	\la_p(U) \ge \la_p(U^*) = \la_p(r B_1) = r^{-p} \la_p(B_1) = \(\frac{|U|}{|B_1|}\)^{-p/n} |B_1|.
\]
\end{proof}

Define
\begin{equation}
\label{m(r)}
	m(r) = \mathcal{H}^{N-1}( \{ x \in \Omega \, | \, |x| = r \}) = \mathcal{H}^{N-1}(\Omega \cap S^{N-1}(r))
\end{equation}
for $r \in (0, R)$. 
Then we have the following. 
%
%
\begin{theorem}
\label{theorem-inequality}
If
\begin{equation}
\label{m_0}
	m_0 \equiv \limsup_{r \to 0} \, m(r)/r^{N-1} < \omega_{N-1}
\end{equation}
and
\begin{equation}
\label{m_R_finite}
	m_R \equiv \limsup_{r \to R} \, m(r)/(R-r)^{N-1} <  \infty,
\end{equation}
it holds that
\[
	C_N(\Omega) > \( \frac{N-1}{N} \)^N.
\]
\end{theorem}

\begin{proof}
If $0 \in \Omega$, then $m(r) = r^{N-1}\omega_{N-1}$ for any small $r > 0$.
Thus under the assumption \eqref{m_0}, the origin must not be interior of $\Omega$.

We assume the contrary and suppose that there exists a sequence $\{\phi_n\}_{n \in \N} $ in $ C_0^\infty(\Omega)\setminus \{0\}$ such that
\[
	\lim_{n \to \infty} \frac{\intO \big|\nabla \phi_n \big|^N dx}{\intO \frac{|\phi_n(x)|^N}{|x|^N \(\log\frac{R}{|x|} \)^N} dx} = C_N(\Omega) = \(\frac{N-1}{N} \)^N.
\]
Let $\phi_n^*$ be the spherical symmetric rearrangement of $\phi_n$.
Then by the above remarks, it follows that
\[
	\lim_{n \to \infty} \frac{\int_{\Omega^*} \big |\nabla \phi^*_n \big|^N dx}{\int_{\Omega^*} \frac{|\phi^*_n(x)|^N}{|x|^N \( \log\frac{R}{|x|} \)^N} dx}
	= C_N(\Omega^*) = \(\frac{N-1}{N} \)^N.
\]
Since $\textrm{supp}(\phi_n^*)$ is compact in $\Omega^*$,
we find positive constants $R_n$ and $\delta_n$ with $\lim_{n \to \infty}R_n$ $=$ $R$ and $\lim_{n \to \infty} \delta_n = 0$ such that
$\textrm{supp}(\phi^*_n) \subset B_{R_n} \setminus \ol{B_{\delta_n}}$.
We define
\[
	\Omega^*_n \equiv \Omega^* \cap (B_{R_n} \setminus \ol{B_{\delta_n}}).
\]
Since the weight function $W_R$ is bounded from above and below by positive constants on $\Omega_n^*$,
there exists a minimizer $\psi_n \in W^{1,N}_0(\Omega^*_n)$ of
\[
	c_n \equiv \inf \Big \{ \int_{\Omega^*_n} \big |\nabla \psi \big|^N dx \ \Big | \
	\int_{\Omega^*_n} \frac{|\psi(x)|^N}{|x|^N \( \log\frac{R}{|x|} \)^N} dx = 1, \, \psi \in W_0^{1,N}(\Omega^*_n) \Big \}.
\]
We may assume $\psi_n \ge 0$, $\psi_n$ satisfies
\[
	\textrm{div}(|\nabla \psi_n|^{N-2}\nabla \psi_n) + c_n \frac{\psi_n(x)^{N-1}}{|x|^N\big (\log\frac{R}{|x|}\big )^N} = 0 \quad  \textrm {in} \ \Omega^*_n,
\]
and $\psi_n$ is rotationally symmetric with respect to $x_N$-axis.
We think that $\psi_n$ is defined on $\Omega^*$ by extending by zero.
Then we see
\begin{equation}
\label{c_n}
	\int_{\Omega^*} |\nabla \psi_n|^{N} dx = c_n \to \Big(\frac{N-1}{N}\Big )^N
\end{equation}
as $n \to \infty$.
Since $\( \frac{N-1}{N} \)^N$ is not attained by any element in $W_0^{1,N}(\Omega^*)$,
elliptic estimates imply that for any small $R^{\prime} > 0$ and any  $\tilde{R} < R$ sufficiently close to $R$,
$\psi_n$ converges uniformly to $0$ on  $\Omega^* \cap (B_{\tilde{R}} \setminus \ol{B_{R^{\prime}}})$
and $\psi_n$ converges weakly to $0$ in $W_0^{1,N}(\Omega^*)$ as $n \to \infty$.
We denote
\[
	\Omega^*(r) \equiv \{ \omega \in S^{N-1} \ | \ r\omega \in \Omega^* \} \subset S^{N-1},
\]
so $m(r) = r^{N-1} \mathcal{H}^{N-1}(\Omega^*(r))$.
Then we note that
\begin{align}
\label{concentration}
	1 &= \int_{\Omega^*} \frac{|\psi_n(x)|^N}{\big (|x|\log\frac{R}{|x|}\big )^N} dx =
	  \int_0^R \int_{\Omega^*(r)} \frac{|\psi_n(r\omega)|^N }{r\big (\log\frac{R}{r}\big )^N} dS_{\omega} dr \notag \\
  	&= \int_0^{R^\prime} \int_{\Omega^*(r)}\frac{|\psi_n(x)|^N}{r\big (\log\frac{R}{r}\big )^N} dS_{\omega} dr
	  + \int_{\tilde{R}}^R \int_{\Omega^*(r)} \frac{|\psi_n(r\omega)|^N }{r\big (\log\frac{R}{r}\big )^N} dS_{\omega} dr + o_n(1)
\end{align}
as $n \to \infty$.

%
%
First, let us assume 
\begin{equation}
\label{concentration on zero}
	\lim_{n \to \infty} \int_0^{R^{\prime}} \int_{\Omega^*(r)} \frac{|\psi_n(r\omega)|^N}{r \(\log\frac{R}{r} \)^N} dS_{\omega} dr \ge C
\end{equation}
for some $C > 0$.
Since $m_0 <\omega_{N-1}$ by assumption \eqref{m_0},
$\Omega^*(r)$ is a proper subset of $S^{N-1} \setminus \{ -e_N \} \simeq \re^{N-1}$ for any small $r >0$.
Thus there exists a constant $C > 0$ independent of small $r > 0$ and $n \in \N$ such that the Poincar\'e inequality
in Proposition \ref{prop-Poincare} (with $U = \Omega^*(r)$, $p = N$, $n = N-1$)
\begin{equation}
\label{Poincare}
	\int_{\Omega^*(r)}|\nabla_{S^{N-1}} \psi_n(r\omega)|^N  dS_{\omega} \ge C \int_{\Omega^*(r)}|\psi_n(r\omega)|^N dS_{\omega}
\end{equation}
holds true.
Note that
\[
	\nabla \psi_n = \frac{x}{|x|}\frac{\partial \psi_n}{\partial r} + \frac{1}{r} \nabla_{S^{N-1}}\psi_n, \qquad
	|\nabla \psi_n|^N \ge \left| \frac{\partial \psi_n}{\partial r} \right|^N + \frac{1}{r^N} |\nabla_{S^{N-1}}\psi_n|^N.
\]
Then for each small $R^\prime > 0$, we have
\begin{align}
\label{poes1}
	\int_{\Omega^*}|\nabla \psi_n|^N dx &= \int_0^R \int_{\Omega^*(r)} \nabla \psi_n(r\omega)|^N r^{N-1} dS_{\omega} dr \notag \\
	&\ge \int_0^{R^\prime} \int_{\Omega^*(r)} \frac{1}{r^{N}} |\nabla_{S^{N-1}} \psi_n|^N r^{N-1} dS_{\omega} dr \notag \\
	&\ge C\int_{0}^{R^\prime} \int_{\Omega^*(r)} \frac{|\psi_n(r\omega)|^N}{r}  dS_{\omega} dr
\end{align}
by the Poincar\'e inequality \eqref{Poincare}.
On the other hand, since
\[
	\int_0^{R^\prime} \int_{\Omega^*(r)}\frac{|\psi_n(r\omega)|^N}{r} dS_{\omega} dr
	\ge \( \log\frac{R}{R^\prime} \)^N
	\int_0^{R^\prime} \int_{\Omega^*(r)} \frac{|\psi_n(r\omega)|^N }{r \(\log\frac{R}{r} \)^N} dS_{\omega} dr,
\]
we have by \eqref{concentration on zero},
\begin{equation}
\label{C2}
	\int_0^{R^\prime} \int_{\Omega^*(r)}\frac{|\psi_n(r\omega)|^N}{r} dS_{\omega} dr \ge \(C + o_n(1) \) \( \log\frac{R}{R^\prime} \)^N
\end{equation}
where $o_n(1) \to 0$ as $n \to \infty$.
Then by \eqref{c_n}, \eqref{poes1}, and \eqref{C2}, we have
\begin{align*}
	\( \frac{N-1}{N} \)^N + o_n(1) &= \int_{\Omega^*}|\nabla \psi_n|^N dx \ge \frac{C}{2} \( \log\frac{R}{R^\prime} \)^N
\end{align*}
as $n \to \infty$.
This inequality is invalid if $R^{\prime}$ is very small.
Thus \eqref{concentration on zero} cannot happen and
\[
	\lim_{n \to \infty} \int_0^{R^{\prime}} \int_{\Omega^*(r)} \frac{|\psi_n(r\omega)|^N}{r \(\log\frac{R}{r} \)^N} dS_{\omega} dr = 0
\]
under the assumption \eqref{m_0}.

%
%
Therefore by \eqref{concentration}, we have
\begin{equation}
\label{concentration on boundary}
	\lim_{n \to \infty} \int_{\tilde{R}}^R \int_{\Omega^*(r)} \frac{|\psi_n(r\omega)|^N}{r \(\log\frac{R}{r} \)^N} dS_{\omega} dr = 1.
\end{equation}

Next, we will prove that \eqref{concentration on boundary} cannot occur under the assumption \eqref{m_R_finite}.
In fact, we see by \eqref{concentration on boundary} and \eqref{Taylor} that
\begin{align*}
	1 + o_n(1) &= \int_{\tilde{R}}^R \int_{\Omega^*(r)} \frac{|\psi_n(r\omega)|^N}{\( r \log \frac{R}{r} \)^N}r^{N-1} dS_{\omega} dr \\
	&= (1+o(1)) R^{N-1}\int_{\tilde{R}}^R \int_{\Omega^*(r)} \frac{|\psi_n(r\omega)|^N}{\( R - r \)^N}dS_{\omega} dr,
\end{align*}
where $o_n(1) \to 0$ as $n \to \infty$ and $o(1) \to 0$ as $\tilde{R} \to R$.
Thus we have
\begin{equation}
\label{ap}
	\lim_{n \to \infty} \int_{\tilde{R}}^R \int_{\Omega^*(r)} \frac{|\psi_n(r\omega)|^N}{\( R - r \)^N}dS_{\omega} dr = (1 + o(1)) R^{-(N-1)}
\end{equation}
as $\tilde{R} \to R$.
On the other hand, since $\psi_n(r\omega) \big|_{r = R} = 0$, we can apply the one-dimensional Hardy inequality
\begin{equation}
\label{1D_Hardy}
	\(\frac{N-1}{N}\)^N \int_{\tilde{R}}^R \frac{|\psi_n(r\omega)|^N}{\(R - r \)^N} dr \le \int_{\tilde{R}}^R  \bigg |\frac{\partial \psi_n(r\omega)}{\partial r}\bigg |^N dr
\end{equation}
to $\psi_n(r\omega)$.
Note that the best constant $\(\frac{N-1}{N}\)^N$ in the inequality \eqref{1D_Hardy} is the same as, by assumption, the value of $C_N(\Omega^*)$.
Then \eqref{1D_Hardy} implies
\begin{align*}
	\(\frac{N-1}{N}\)^N \int_{\tilde{R}}^R \int_{\Omega^*(r)} \frac{|\psi_n(r\omega)|^N}{\(R - r \)^N}dS_{\omega} dr
	&\le \int_{\tilde{R}}^R \int_{\Omega^*(r)} \bigg |\frac{\partial \psi_n}{\partial r}(r\omega) \bigg |^N dS_{\omega} dr \\
	&= (1+o(1)) R^{-(N-1)} \int_{\Omega^*} \bigg|\frac{\partial \psi_n}{\partial r}(x) \bigg|^N dx.
\end{align*}
The above inequality, \eqref{ap} and $C_N(\Omega^*) = (\frac{N-1}{N})^N  = \lim_{n \to \infty} \int_{\Omega^*} |\nabla \psi_n(x)|^N dx$ by \eqref{c_n}
imply that
\[
	\lim_{n \to \infty} \int_{\Omega^*}|\nabla \psi_n|^N dx \le \lim_{n \to \infty} \int_{\Omega^*}\bigg|\frac{\partial \psi_n}{\partial r}(x) \bigg|^N dx.
\]
The converse inequality holds trivially, thus we see that
\[
	\lim_{n \to \infty} \int_{\Omega^*}|\nabla \psi_n|^N dx
	= \lim_{n \to \infty} \int_{\Omega^*}\bigg|\frac{\partial \psi_n}{\partial r}\bigg|^N dx,
\]
which implies
\begin{equation}
\label{angle vanish}
	\lim_{n \to \infty} \int_{R^\prime}^R \int_{r \Omega^*(r)} |\nabla_{S^{N-1}(r)}\psi_n(\sigma)|^N| d\sigma_r dr = 0,
\end{equation}
here $\sigma = r\omega \in S^{N-1}(r)$, $d\sigma_r = r^{N-1} dS_{\omega}$ is a volume element of a geodesic ball $r\Omega^*(r)$ with center $r e_N$ in $S^{N-1}(r)$,
and $\nabla_{S^{N-1}(r)} = (1/r) \nabla_{S^{N-1}}$.

%
%
From the assumption $m_R < \infty$ in \eqref{m_R_finite}, there exists a constant $ C > 0$ independent of $r \in (\tilde{R}, R)$ and $n$ such that
\[
	r^{N-1} \mathcal{H}^{N-1}(\Omega^*(r)) \le C(R-r)^{N-1}
\]
holds true.
This implies that \[\( \mathcal{H}^{N-1}(r \Omega^*(r)) \)^{-N/(N-1)}  \ge  D (R - r)^{-N},\]
where $D = C^{-N/(N-1)} > 0$ independent of $r \in (\tilde{R}, R)$ and $n$.
Then, by the Poincar\'e inequality in Proposition \ref{prop-Poincare} ($n = N-1$, $p = N$) on the spherical cap $U = r \Omega^*(r) \subset S^{N-1}(r)$,
\begin{equation}
\label{Poincare2}
	\int_{r\Omega^*_r}  |\nabla_{S^{N-1}(r)}\psi_n(\sigma)|^N d\sigma_r \ge  D \int_{r\Omega^*_r} \frac{|\psi_n(\sigma)|^N}{|R-r|^N} d\sigma_r
\end{equation}
holds true.
Combining \eqref{angle vanish} and \eqref{Poincare2}, we have
\begin{align*}
	o_n(1) &= \int_{\tilde{R}}^R \int_{r\Omega^*(r)} |\nabla_{S^{N-1}(r)}\psi_n(\sigma)|^N d\sigma_r dr \ge D \int_{\tilde{R}}^R \int_{r \Omega^*(r)} \frac{|\psi_n(\sigma)|^N}{|R-r|^N} d\sigma_r dr \\
	&= (1+o(1)) D R^{N-1}  \int_{\tilde{R}}^R \int_{\Omega^*(r)} \frac{|\psi_n(r\omega)|^N}{\(R - r \)^N} dS_{\omega} dr
\end{align*}
where $o_n(1) \to 0$ as $n \to \infty$ and $o(1) \to 0$ as $\tilde{R} \to R$.
Combining this to \eqref{ap} and letting $n \to \infty$, we see
\[
	0 = D (1 + o(1))R^{N-1} \times (1 + o(1)) R^{-(N-1)} = D + o(1)
\]
as $\tilde{R} \to R$.
This is a contradiction and we complete the proof.
\end{proof}

%
%

Next, we prove that a condition on $\Omega$ stronger than that of in Theorem \ref{theorem-inequality} assures the attainability of $C_N(\Omega)$.
The condition below implies that the boundary point $x \in \partial B_R \cap \pd\Omega$, if it existed, must be cuspidal,
but the origin, if $0 \in \partial \Omega$, may be a Lipschitz continuous boundary point.

\begin{theorem}
\label{theorem-existence}
For $r \in (0,R)$, let $m(r)$ be defined as \eqref{m(r)}.
If
\[
	m_0 \equiv \limsup_{r \to 0} \, m(r)/r^{N-1} < \omega_{N-1}
\]
and
\begin{equation}
\label{m_R_0}
	m_R \equiv \limsup_{r \to R} \, m(r)/(R-r)^{N-1} = 0,
\end{equation}
then
\[
	C_N(\Omega) > \( \frac{N-1}{N} \)^N
\]
and $C_N(\Omega)$ is attained.
\end{theorem}

\begin{proof}
The strict inequality $C_N(\Omega) > \( \frac{N-1}{N} \)^N $ was proved in Theorem \ref{theorem-inequality}.

For each positive integer $n$, we define
\[
	\Omega_n \equiv \Omega \cap (B_{R-1/n} \setminus \overline{B_{1/n}}).
\]
Then, since the weight function $W_R(x)$ is bounded on $\Omega_n$, there exists a minimizer $\psi_n$ of
\[
	d_n \equiv \inf \Big \{ \int_{\Omega_n} \big |\nabla \psi\big|^N dx \ \Big | \
	\int_{\Omega_n} \frac{|\psi(x)|^N}{|x|^N \(\log\frac{R}{|x|}\)^N} dx = 1, \, \psi \in W_0^{1,N}(\Omega_n) \Big \}.
\]
We may assume $\psi_n \ge 0$ and $\psi_n$ satisfies
\[
	\textrm{div}(|\nabla \psi_n|^{N-2}\nabla \psi_n) +
	d_n \frac{\psi_n(x)^{N-1}}{|x|^N \(\log\frac{R}{|x|} \)^N} = 0 \ \ \textrm { in } \ \Omega_n.
\]
We note that
\[
	\int_{\Omega_n} |\nabla \psi_n|^{N} dx =  d_n \to C_N(\Omega) \ \textrm { as } \ n \to \infty.
\]
Let $u$ be a weak limit of the sequence $\{\psi_n\}_{n \in \N}$ in $W_0^{1,N}(\Omega)$.
Then, we see that for each positive integer $n_0$,
$\psi_n$ converges uniformly to $u$ in $C^{1}(\Omega_{n_0})$,
and that
\[
	\textrm{div}(|\nabla u|^{N-2}\nabla u) +
	C_N(\Omega) \frac{|u(x)|^{N-1}}{|x|^N \(\log\frac{R}{|x|} \)^N} = 0, \ \ u \ge 0  \ \  \textrm {in} \ \Omega.
\]
Now it suffices to prove that $u \ne 0$ in $\Omega$, then $u$ becomes a minimizer for $C_N(\Omega)$.

To the contrary, we assume that $u \equiv 0$.
Then, we see that for each positive integer $n_0$, $\psi_n$ converges uniformly to $0$ on $\Omega_{n_0}$.
We denote
\[
	\Omega(r) \equiv \{ \omega \in S^{N-1} \ | \ r\omega \in \Omega \} \subset S^{N-1}.
\]
Since $m_0 <\omega_{N-1}$,
by the spherical symmetric rearrangement, Poly\'a-Szeg\"o and the Poincar\'e inequality,
we see there exists a constant $C > 0$, independent of small $r > 0$ and $n \in \N$, such that
\[
	\int_{\Omega(r)}|\nabla_{S^{N-1}} \psi_n|^N  dS_{\omega} \ge C\int_{\Omega(r)}|\psi_n|^N dS_{\omega},
\]
see the proof of Theorem \ref{theorem-inequality}.
Then, we see that for each large positive integer $n_0$,
\begin{align}
\label{poes2}
	\intO |\nabla \psi_n|^N dx &\ge \int_0^{1/n_0} \int_{\Omega(r)} |\nabla_{S^{N-1}} \psi_n(r\omega)|^N r^{-1} dS_{\omega} dr \notag \\
	&\ge C\int_{0}^{1/n_0} \int_{\Omega(r)}|\psi_n(r\omega)|^N r^{-1} dS_{\omega} dr.
\end{align}
Put $f_n(r) \equiv \int_{\Omega(r)} |\psi_n(r\omega)|^N / r \(\log\frac{R}{r} \)^N dS_{\omega}$.
Then we have
\begin{align*}
	1 = &\intO \frac{|\psi_n(x)|^N}{\(|x|\log\frac{R}{|x|}\)^N} dx = \int_0^R \int_{\Omega(r)} \frac{|\psi_n(r\omega)|^N }{r \(\log\frac{R}{r} \)^N} dS_{\omega} dr \\
	& = \int_0^{1/n_0} f_n(r) dr + \int_{1/n_0}^{R-1/n_0} f_n(r) dr +  \int_{R-1/n_0}^{R} f_n(r) dr,
\end{align*}
and that
\begin{align*}
	&\int_0^{1/n_0} \int_{\Omega(r)} \frac{|\psi_n(r\omega)|^N }{r \(\log\frac{R}{r} \)^N} dS_{\omega} dr
	\le  \( \log\frac{R}{1/n_0} \)^{-N} \int_0^{1/n_0} \int_{\Omega(r)}\frac{|\psi_n(r\omega)|^N}{r} dS_{\omega} dr.
\end{align*}
Then, \eqref{poes2} implies that for each large positive integer $n_0$,
\[
	\int_0^{1/n_0} \int_{\Omega(r)} \frac{|\psi_n(r\omega)|^N }{r \(\log\frac{R}{r} \)^N} dS_{\omega} dr
	\le  \big (\log\frac{R}{1/n_0}\big )^{-N} \frac{d_n}{C}.
\]
The right-hand side of the above inequality can be arbitrarily small if $n_0$ large, thus we have
$\lim_{n \to \infty} \int_0^{1/n_0} f_n(r) dr = 0$.
Since $\lim_{n \to \infty}  \int_{1/n_0}^{R-1/n_0} f_n(r) dr = 0$,
we deduce that for each large positive integer $n_0$,
\[
	\lim_{n \to \infty} \int_{R-1/n_0}^R f_n(r) dr = 1.
\]

Now, as in the proof of Theorem \ref{theorem-inequality},
let $\Omega^*(r) \subset S^{N-1}$ be a geodesic ball with the center $e_N$ such that the $(N-1)$-dimensional measure of $\Omega^*(r)$ equals to that of $\Omega(r)$.
Let $\psi^*_n$ be the spherical symmetric rearrangement of $\psi_n$ and
put $f^*_n(r) = \int_{\Omega^*(r)} \frac{|\psi^*_n(r\omega)|^N }{r\big (\log\frac{R}{r}\big )^N} dS_{\omega}$.
Since $r\log(R/r) = (R-r) + o(1)$ for small $R - r > 0$,
we see that
\begin{equation}
\label{E1}
	f^*_n(r) = \int_{\Omega^*(r)} \frac{|\psi^*_n(r\omega)|^N }{r\big (\log\frac{R}{r}\big )^N} dS_{\omega} = R^{N-1} \int_{\Omega^*(r)} \frac{|\psi^*_n(r\omega)|^N }{(R-r)^N} dS_{\omega} + o(1)
\end{equation}
for small $R - r > 0$.

%
%
On the other hand, by the assumption $m_R = 0$,
there exists $h(r) > 0$ with $h(r) \to 0$ as $r \to R$ such that $\mathcal{H}^{N-1}(r \Omega^*(r)) \le h(r) (R - r)^{N-1}$.
Thus
\[
	\( \mathcal{H}^{N-1}(\Omega^*(r)) \)^{-N/(N-1)} \ge r^N \( h(r) \)^{-N/(N-1)} (R - r)^{-N}.
\]
Put $g(r) = r^N ( h(r) )^{-N/(N-1)}$. Then $\lim_{r \to R} g(r) = \infty$ and the Poincar\'e inequality in Proposition \ref{prop-Poincare}
(with $U = \Omega^*(r)$, $p = N$, $n = N-1$)
\begin{equation}
\label{E2}
	\int_{\Omega^*(r)} |\nabla_{S^{N-1}}\psi^*_n(r\omega)|^N dS_{\omega} \ge C g(r)\int_{\Omega^*(r)}  \frac{|\psi^*_n(r\omega)|^N}{|R-r|^N} dS_{\omega}
\end{equation}
holds. Here $C = C(N) >0$ is an absolute constant.
Then by \eqref{E1} and \eqref{E2}, we see
\[
	\int_{\Omega^*(r)} |\nabla_{S^{N-1}}\psi^*_n(r\omega)|^N dS_{\omega} \ge \frac{C}{2} g(r) \frac{f^*_n(r)}{R^{N-1}}
\]
and we may apply Poly\'a-Szeg\"o inequality
\[
	\int_{\Omega(r)} |\nabla_{S^{N-1}} \psi_n(r\omega)|^N dS_{\omega} \ge \int_{\Omega^*(r)} |\nabla_{S^{N-1}} \psi^*_n(r\omega)|^N dS_{\omega}.
\]
Then for large $n_0$, we have
\begin{align*}
	&\intO |\nabla \psi_n|^N dx \ge \int_{R - 1/n_0}^R \int_{\Omega(r)} |\nabla_{S^{N-1}} \psi_n(r\omega)|^N dS_{\omega} dr \\
	&\ge \int_{R-1/n_0}^{R} \frac{C}{2} \frac{g(r) f^*_n(r)}{R^{N-1}} dr \ge \frac{C g(r^*)}{2 R^{N-1}} \int_{R-1/n_0}^{R} f^*_n(r) dr
	= \frac{C g(r^*)}{2 R^{N-1}} (1 + o_n(1))
\end{align*}
where $r^*$ is a number with $r^* \in (R-1/n_0, R)$.
Since $g(r^*) \to \infty$ as $n_0 \to \infty$,
we conclude that $\lim_{n \to \infty}\int_{\Omega}|\nabla \psi_n|^N dx  = \infty$.
This is a contraction; thus $C_N(\Omega)$ is attained.

\end{proof}

\end{section}

%
%
\begin{section}{Nonexistence of a minimizer for a domain $\Omega$ with $C_2(\Omega) >  \frac{1}{4} $}

In this section, we provide a Lipschitz domain $\Omega$ in $\re^2$ on which $C_2(\Omega) > 1/4$ and $C_2(\Omega)$ is not attained. Recall Hardy's inequality \eqref{Hardy_half_inf} when $N = 2$:
\[
	\inf \left\{ \int_{\re^2_{+}} |\nabla u|^2 dx \ \Big | \
	\int_{\re^2_{+}} \frac{u^2}{(x_2)^2} dx = 1, \, u \in W_0^{1,2}(\re^2_+) \right\} = \frac {1}{4},
\]
and the best constant $1/4$ is not attained, where $x = (x_1, x_2)$.
For $a \in [0,\pi/2),$ we define
\[
	E(a) \equiv \inf \Big \{ \frac{\int_{a}^{\pi-a}(\phi_\theta)^2 d\theta} {\int_{a}^{\pi-a} (\phi^2 / \sin^2 \theta) d\theta}
	\ \Big  | \ \phi \in C_0^\infty((a,\pi-a)) \setminus \{ 0 \} \Big \}.
\]

From \cite[Corollary 4.4]{Davies}, we see that
\begin{equation}
\label{Davies}
	E \equiv E(0) = \inf \Big \{ \frac{\int_{0}^{\pi}(\phi_\theta)^2 d\theta} {\int_{0}^{\pi} (\phi^2/ \sin^2 \theta) d\theta}
	\ \Big | \ \phi \in C_0^\infty((0,\pi)) \setminus \{ 0 \} \Big \} = \frac {1}{4}
\end{equation}
and $E$ is not achieved.
We prove these facts in Appendix for the reader's convenience.
It is obvious that for $a \in (0,\pi/2),$ $E(a)$ is achieved by a positive function $\varphi_a$ on $(a,\pi-a).$
Since $E(0)$ is not achieved in $W^{1,2}_0(0,\pi),$ $E(a) > E(0) = \frac 14$ for $a \in (0,\pi/2).$
%
%
\begin{theorem}
\label{theorem-nonexistence}
There exists a domain $\Omega \subset B_1 \subset \re^2$
such that
$C_2(\Omega) > \frac{1}{4}$ and $C_2(\Omega)$ is not attained.
\end{theorem}
\begin{proof}
For $a \in (0,\pi/2)$, we define a cone
\[
	\mathbf{C}_a \equiv \{ (r\cos \theta, r\sin \theta) \in \re^2_+ \ | \ r \in (0,\infty), \theta \in (a,\pi-a)\} \subset \re^2_+.
\]
We define
\begin{align*}
	R(y_1,y_2) &\equiv \((y_1)^2 + (1-y_2)^2 \) \( \log \frac{1}{((y_1)^2 + (1 -y_2)^2)^{1/2}} \)^2 \\
	&= \frac{1}{4} h(r, \theta) \{ \log h(r,\theta) \}^2
\end{align*}
for $(y_1, y_2) = (r \cos \theta, r \sin \theta)$, where $h(r, \theta) = r^2 - 2 r \sin \theta + 1$.
Since \[ \log h(r, \theta) = h(r,\theta) - 1 - \frac{(h(r, \theta) - 1)^2}{2} + O(r^3) \textrm{  as }  r \to 0, \]
we have
\begin{align}
\label{asymp}
	\frac{R(y_1, y_2)}{(y_2)^2} &= \frac{(r^2 - 2r \sin \theta + 1) (4\sin^2 \theta - 4r \sin \theta (1 - 2\sin^2 \theta) + O(r^2))}{4 \sin^2 \theta} \notag \\
	&= \frac{4 \sin^2 \theta - 4r \sin \theta + O(r^2)}{4 \sin^2 \theta}
\end{align}
as $r \to 0$.
Thus we see that
\[
	\lim_{y_2 \to 0, (y_1,y_2) \in \mathbf{C}_a} R(y_1,y_2)/(y_2)^2 = 1
\]
for each $a > 0$.
From now on, we fix $a \in (\pi/4,\pi/2)$.
We define
\[
	g(r) \equiv  \inf \Big \{ \frac{R(y_1,y_2)}{(y_2)^2} \ \Big | \ (y_1,y_2) \in \mathbf{C}_a, \,  y_1^2+y_2^2 = r^2 \Big \}.
\]
By \eqref{asymp}, we see that $\lim_{r \to 0}g(r) = 1$.
Further, we see  that $g(r) < 1$ for small $r > 0$.
We take $r_0 \in (0,1/2)$ such that $g(r) < 1$ for any $r \in (0,r_0)$.
Note that $E(a)$ is monotone non-decreasing with respect to $a \in (0, \pi/2)$.
Now for each $r \in (0,r_0)$, we take $a(r) \in (a,\pi/2)$ such that
$E(a)/E(a(r)) = g(r) \in (0,1)$.
Since $\lim_{r \to 0}g(r) =1$, it follows that $\lim_{r \to 0} a(r) = a$.
Since $E$ is continuous on $(0,\pi/2)$ and $g$ on $(0,r_0)$,
$a(r)$ is continuous with respect to $r \in (0,r_0)$.
We define
\[
	\tilde{\Omega} \equiv  \{(r\cos\theta,r\sin\theta) \in \re^2_+ \ | \ r \in (0, r_0), \theta \in (a(r),\pi-a(r))\}
\]
and
\[
	\Omega = \{(x_1,x_2) \in B_1 \ | \ (x_1,1-x_2) \in \tilde{\Omega} \} \subset B_1 \subset \re^2.
\]
We claim that $C_2(\Omega) = E(a) >\frac14 $ and $C_2(\Omega)$ is not attained.

For any $u \in C_0^\infty(\Omega),$ we define $\tilde{u}(y_1,y_2) = u(y_1,1-y_2)$ for $y = (y_1, y_2) \in \tilde{\Omega}$.
Then, we see that $\tilde{u} \in C_0^\infty(\tilde{\Omega})$ and
\[
	\int_{\Omega} |\nabla u|^2 dx_1dx_2 = \int_{\tilde{\Omega}} |\nabla \tilde{u}|^2 dy_1dy_2
	= \int_0^{r_0} \int_{a(r)}^{\pi-a(r)} r(\tilde{u}_r)^2 + r^{-1}(\tilde{u}_\theta)^2 d\theta dr
\]
and
\[
	\int_{\Omega} \frac{(u(x_1,x_2))^2}{|x|^2(\log |x|)^2} dx_1dx_2 = \int_{\tilde{\Omega}} \frac{(\tilde{u}(y_1,y_2))^2}{ R(y_1,y_2) } dy_1dy_2.
\]

First of all, we claim that $C_2(\Omega) \le E(a)$.
To prove this, we note that for any $a^\prime \in (a,\pi/2)$, we can find $\delta^\prime \in (0,r_0)$ such that
\[
	\{(r\cos\theta,r\sin \theta) \in \tilde{\Omega} \ | \ r \in (0,\delta^\prime), \theta \in (a^\prime,\pi-a^\prime)\} \subset \tilde{\Omega}.
\]
For any small $\e,\delta > 0$ with  $4\e <  \delta < \delta^\prime$,
we find a Lipschitz continuous function $\psi_\e^\delta$ satisfying
$\psi_\e^\delta(r) = 0$ for $r \le \e$  or $r \ge \delta$, $\psi_\e^\delta(r) = 1$ for $2\e \le r \le \delta/2$,
$|(\psi_\e^\delta)^\prime(r)| = 1/\e$ for $r \in (\e,2\e)$, and $|(\psi_\e^\delta)^\prime(r)| = 2/\delta$ for $r \in (\delta/2, \delta)$.
We define that for $y=(y_1,y_2) = (r \cos \theta,r\sin\theta) \in \tilde{\Omega}$ and $x=(x_1,x_2) \in \Omega$,
 \[
\tilde{u}^{\delta}_{\e}(y_1,y_2)  = \tilde{u}^{\delta}_{\e}(r,\theta) = \psi_\e^\delta(r)\varphi_{a^\prime}(\theta) \textrm { and }
	u^{\delta}_{\e}(x_1,x_2) = \tilde{u}^{\delta}_{\e}(x_1,1-x_2).
\]
Then we see that
\begin{align*}
	&\int_{\Omega} |\nabla u^{\delta}_{\e}|^2 dx = \int_{\tilde{\Omega}} |\nabla\tilde{u}^{\delta}_{\e} |^2 dy
= \int_0^\infty \int_{a^\prime}^{\pi-a^\prime} r((\tilde{u}^{\delta}_{\e})_r)^2 + r^{-1}((\tilde{u}^{\delta}_{\e})_{\theta})^2 d\theta dr \\
	& = \( \int_\e^{2\e} ((\psi_{\e}^\delta)^{\prime}(r))^2 r dr + \int_{\delta/2}^{\delta} ((\psi_{\e}^\delta)^{\prime}(r))^2 r dr \) \int_{a^\prime}^{\pi-a^\prime}(\varphi_{a^\prime}(\theta))^2d\theta \\
	&\quad  + \int_\e^{\delta}\int_{a^\prime}^{\pi-a^\prime} r^{-1}(\psi_\e^\delta(r))^2\Big(\frac{d\varphi_{a^\prime}}{d\theta}\Big)^2 d\theta dr \\
	& = 3\int_{a^\prime}^{\pi-a^\prime}(\varphi_{a^\prime}(\theta))^2d\theta + \int_\e^{\delta} r^{-1}(\psi_\e^\delta(r))^2 dr\int_{a^\prime}^{\pi-a^\prime}\Big(\frac{d\varphi_{a^\prime}}{d\theta}\Big )^2 d\theta
\end{align*}
and
\[
	\int_{\Omega} \frac{(u^{\delta}_{\e}(x))^2}{|x|^2(\log |x|)^2} dx = \int_{\tilde{\Omega}}\frac{(\tilde{u}^{\delta}_{\e}(y))^2}{R(y_1,y_2)} dy = \int_\e^{\delta}\int_{a^\prime}^{\pi-a^\prime}\frac{(y_2)^2}{R(y_1,y_2)} r^{-1}(\psi_\e^\delta(r))^2 \Big(\frac{\varphi_{a^\prime}}{\sin \theta}\Big )^2 d\theta dr.
\]
Since $\lim_{\e \to 0}\int_\e^{\delta} r^{-1}(\psi^\delta_\e(r))^2 dr = \infty$ for each $\delta > 0,$
we see that
\[
	\lim_{\e \to 0} \frac{\int_{\Omega}|\nabla u^\delta_\e|^2 dx}{ \int_{\Omega}  \frac{|u^\delta_\e|^2}{|x|^2(\log |x|)^2} dx } \le
E(a^\prime) (\min_{r \in [0,\delta]}g(r))^{-1}.
\]
Then, $C_2(\Omega) \le E(a^\prime)$ for any $a^\prime \in (a, \pi/2)$ since $\lim_{r \to 0} g(r) = 1$.
This implies that $C_2(\Omega) \le E(a)$.

Now for any $v \in W_0^{1,2}(\Omega)$ with $\tilde{v}(y_1,y_2) \equiv v(y_1,1-y_2) \in W_0^{1,2}(\tilde{\Omega}),$
 we see that
\begin{align*}
	\int_{\Omega} |\nabla v|^2 dx_1dx_2
& \ge \int_0^{r_0} \int_{a(r)}^{\pi-a(r)} r(\tilde{v}_r)^2 + E(a(r))r^{-1}\frac{(\tilde{v})^2}{\sin^2 \theta} d\theta dr  \\
& = \int_0^{r_0} \int_{a(r)}^{\pi-a(r)} \Big [ (\tilde{v}_r)^2 + E(a(r))\frac{(\tilde{v})^2}{(y_2)^2}\Big] rd\theta dr \\
&= \int_0^{r_0} \int_{a(r)}^{\pi-a(r)} \Big [ (\tilde{v}_r)^2 + E(a(r)) \frac{R(y_1,y_2)}{(y_2)^2}\frac{(\tilde{v})^2}{R(y_1,y_2)} \Big ] rd\theta dr \\
&  \ge \int_0^{r_0} \int_{a(r)}^{\pi-a(r)} \Big [ (\tilde{v}_r)^2 + E(a(r)) g(r)\frac{(\tilde{v})^2}{R(y_1,y_2)} \Big ] rd\theta dr \\
& =   \int_0^{r_0} \int_{a(r)}^{\pi-a(r)} \Big [ (\tilde{v}_r)^2 + E(a) \frac{(\tilde{v})^2}{R(y_1,y_2)} \Big ] rd\theta dr \\
& =   \int_0^{r_0} \int_{a(r)}^{\pi-a(r)}  (\tilde{v}_r)^2 rd\theta dr + E(a)  \int_{\tilde{\Omega}} \frac{(\tilde{v})^2}{R(y_1,y_2)} dy_1dy_2\\
& =   \int_0^{r_0} \int_{a(r)}^{\pi-a(r)}  (\tilde{v}_r)^2 rd\theta dr + E(a)  \int_{\Omega} \frac{(v(x))^2}{|x|^2(\log |x|)^2} dx.
\end{align*}
This implies that $C_2(\Omega) \ge E(a).$
Combining above upper and lower estimates, we see that $C_2(\Omega) = E(a) > \frac14.$

From above estimate, we see that
 for any $u \in W_0^{1,2}(\Omega),$ we see that
\begin{equation} \label{ces} \int_{\Omega} |\nabla u|^2 dx_1dx_2   \ge    \int_0^{r_0} \int_{a(r)}^{\pi-a(r)}  (\tilde{u}_r)^2 rd\theta dr + E(a)  \int_{\Omega} \frac{(u(x_1,x_2))^2}{|x|^2(\log |x|)^2} dx_1dx_2.\end{equation}
If $C_2(\Omega)$ is attained by $u \in W_0^{1,2}(\Omega) \setminus \{0\},$
we see from \eqref{ces} that $\tilde{u}_r \equiv 0$ in $\tilde{\Omega}$.
This contradicts to the fact $u \in W_0^{1,2}(\Omega)$.
Thus we conclude that $C_2(\Omega)$ is not attained in $W_0^{1,2}(\Omega)$.
\end{proof}

$\Box$

%
%
\begin{remark}
For the domain $\Omega$ in Theorem \ref{theorem-nonexistence},
let $P, Q$ be two points in $\pd \Omega \cap \pd B_r$ when $r$ is close to $1$.
Then $m(r)$ is the length of the arc $\stackrel{\frown}{PQ}$, which is larger than the length of the segment $PQ$.
Thus it is easy to see that in this case, $m_0 = 0$ and
\[
	m_1 = \limsup_{r \to 1} m(r)/(1-r) \ge 2\cos a > 0;
\]
see Theorem \ref{theorem-existence}.
\end{remark}

\end{section}

\appendix
\section{Appendix}

Here, we prove
\[
	E \equiv \inf \Big \{ \frac{\int_{0}^{\pi}(\phi_\theta)^2 d\theta} {\int_{0}^{\pi} (\phi^2/ \sin^2 \theta) d\theta}
	\ \Big | \ \phi \in W^{1,2}_0(0,\pi) \setminus \{ 0 \} \Big \} = \frac {1}{4}
\]
and $E$ is not achieved.

\begin{proof}
For $u \in C_0^{\infty}((0,\pi))$, we compute
\begin{align*}
	&\left| \int_0^\pi \frac{u^2}{\sin^2 \theta} d\theta \right| = \left| \int_0^\pi \( -\frac{\cos \theta}{\sin \theta} \)^{\prime} u^2 d\theta \right|
	= \left| \int_0^\pi \( \frac{\cos \theta}{\sin \theta} \) 2u u^{\prime} d\theta \right| \\
	&\le 2 \(\int_0^\pi \frac{u^2}{\sin^2 \theta} d\theta \)^{\frac 12} \(\int_0^\pi (u^\prime)^2 \cos^2 \theta d\theta \)^{\frac 12}
	\le 2 \(\int_0^\pi \frac{u^2}{\sin^2 \theta} d\theta \)^{\frac 12} \(\int_0^\pi (u^\prime)^2 d\theta \)^{\frac 12}.
\end{align*}
Thus we have the inequality
\[
	\frac{1}{4} \int_0^\pi \frac{u^2}{\sin^2 \theta} d\theta \le \int_0^\pi (u^\prime)^2 d\theta.
\]
By density, this inequality holds for all $u \in W^{1,2}_0(0, \pi)$.

To see $E = 1/4$, test $E$ by functions $u_{\alpha}(\theta) = (\sin \theta)^{\alpha}$ for $\alpha > 1/2$.
Then we find
\[
	\frac{\int_{0}^{\pi}(u_{\alpha}^{\prime}(\theta))^2 d\theta}{\int_{0}^{\pi} (u_{\alpha}^2 / \sin^2 \theta) d\theta}
	= \alpha^2 - \frac{\int_{0}^{\pi} (\sin \theta)^{2\alpha-2} d\theta}{\int_{0}^{\pi} (u_{\alpha}^2 / \sin^2 \theta) d\theta}
	\le \alpha^2 \to 1/4, \quad \alpha \downarrow 1/2.
\]

To see that $E$ is not attained, we use the function $v(\theta) = u(\theta)/(\sin \theta)^{1/2}$ for $u \in W^{1,2}_0(0,\pi)$.
Then a simple computation shows that
\[
	(u^{\prime})^2 - \frac{1}{4} \frac{u^2}{\sin^2 \theta} = -\frac{u^2}{4} + (v^{\prime})^2 \sin \theta + \(\frac{v^2}{2}\)^{\prime} \cos \theta.
\]
Integrating this on $[0, \pi]$, and noting that $\int_0^\pi (v^2/2)^{\prime} \cos \theta d\theta = \int_0^\pi (u^2/2) d\theta$ by integration by parts,
we obtain
\[
	\int_0^\pi \left[ (u^{\prime})^2 - \frac{1}{4} \frac{u^2}{\sin^2 \theta} \right] d\theta = \int_0^\pi \frac{u^2}{4} d\theta  + \int_0^\pi (v^{\prime})^2 \sin \theta d\theta.
\]
This implies that if $E$ is attained, then $u \equiv 0$ on $[0, \pi]$.
\end{proof}

%
%

\vspace{1em}\noindent
{\bf Acknowledgments.}

This research of the first author(J.B.) was supported by
Mid-career Researcher Program through the National Research Foundation
of Korea funded by the Ministry of Science, ICT and Future Planning
(NRF-2017R1A2B4007816).
The second author (F.T.) was supported by JSPS Grant-in-Aid for Scientific Research (B), No.15H03631.

\end{document}